\documentclass[10pt]{article}
\usepackage[T1]{fontenc}
\usepackage[cp1250]{inputenc}
\usepackage{lmodern}

\usepackage[english]{babel}
\usepackage{latexsym}
\usepackage{amsmath}
\usepackage{indentfirst}
\usepackage{amsthm}
\usepackage{amssymb}
\usepackage{mathabx}
\usepackage{amsfonts}
\usepackage{stackrel}
\usepackage{dsfont}
\usepackage{tikz}
\usepackage{slashbox}
\usepackage[mathscr]{eucal}
\usepackage{authblk}
\usepackage{hyperref}

\newtheorem{thm}{Theorem}
\newtheorem{lemma}[thm]{Lemma}
\newtheorem{cor}[thm]{Corollary}

\newtheorem{defin}{Definition}

\newcommand{\reals}{\mathbb{R}}
\newcommand{\naturals}{\mathbb{N}}

\newcommand{\complex}{\mathbb{C}}

\newcommand{\eps}{\varepsilon}
\newcommand{\supp}{\text{supp}}

\newcommand{\Spher}{\mathcal{S}}
\newcommand{\KGK}{K\backslash G/K}
\newcommand{\X}{\mathcal{X}}

\begin{document}
\title{Sobolev spaces on Gelfand pairs}
\author{Mateusz Krukowski}
\affil{Institute of Mathematics, \L\'od\'z University of Technology, \\ W\'ol\-cza\'n\-ska 215, \
90-924 \ \L\'od\'z, \ Poland \\ \vspace{0.3cm} e-mail: mateusz.krukowski@p.lodz.pl}
\maketitle

\begin{abstract}
The primary aim of the paper is the study of Sobolev spaces in the context of Gelfand pairs. The article commences with providing a historical overview and motivation for the researched subject together with a summary of the current state of the literature. What follows is a general outline of harmonic analysis on Gelfand pairs, starting with a concept of positive-semidefinite functions, through spherical functions and spherical transform and concluding with the Hausdorff-Young inequality. The main part of the paper introduces the notion of Sobolev spaces on Gelfand pairs and studies the properties of these spaces. It turns out that Sobolev embedding theorems and Rellich-Kondrachov theorem still hold true in this generalized context (if certain technical caveats are taken into consideration). 
\end{abstract}

\vspace{0.2cm}

\section{Introduction}

It is difficult to imagine a mathematician working in a field of differential equations, who has not heard of Sobolev spaces. These spaces are named after Sergei Sobolev ($1908-1989$), although they were known before the rise of the Russian mathematician to academic stardom. In 1977 Gaetano Fichera wrote (comp. \cite{Graffi}):
``These spaces, at least in the particular case $p=2$, were known since the very beginning of this century, to the Italian mathematicians Beppo Levi and Guido Fubini who investigated the Dirichlet minimum principle for elliptic equations.'' According to Fichera, at the beginning of the fifties, a group of French mathematicians decided to dub the spaces in question and they came up with the name ``Beppo Levi spaces''. Such a choice, however, was frowned up by Beppo Levi ($1875-1961$) himself, so the name had to be changed. Eventually, the spaces were named after Sergei Sobolev and the rest is history...  

Today, Sobolev spaces are the subject of countless papers, articles and monographs. It is a futile effort to try to list them all and thus we restrict ourselves to just a handful of selected examples. Among the most popular (in the author's personal opinion) are Brezis' ``Functional Analysis, Sobolev Spaces and Partial Differential Equations'' (comp. \cite{Brezis}) or Leoni's ``A First Course in Sobolev spaces'' (comp. \cite{Leoni}). These are by no means the only worth-reading monographs $-$ we should also mention Adams' and Fournier's ``Sobolev spaces'' (comp. \cite{AdamsFournier}),  Fran\c{c}oise and Gilbert Demengel's ``Functional Spaces for the Theory of Elliptic Partial Differential Equations'' (comp. \cite{Demengels}), Evans' ``Partial Differential Equations'' (comp. \cite{Evans}), Necas' ``Direct Methods in the Theory of Elliptic Equations'' (comp. \cite{Necas}), Tartar's ``An Introduction to Sobolev Spaces and Interpolation Spaces'' (comp. \cite{Tartar}) and many more. 

Although more than 70 years have passed since the birth of Sobolev spaces, they still remain an active field of research. Mathematicians study Sobolev spaces on time scales (comp. \cite{AgarwalOterEspinarPereraVivero, FengSuYao, LiZhou, AhmadBaigRahmanSaleem}) as well as the Sobolev spaces with variable exponent (comp. \cite{EdmundsRakosnik1, EdmundsRakosnik2, KovacikRakosnik, MihailescuRadulescu, SamkoVakulov}). Plenty of papers is devoted to fractional Sobolev spaces (comp. \cite{BahrouniRadulescu, MazyaShaposhnikova, NezzaPalatucciValdinoci, Strichartz, Swanson}). Even at the author's alma mater, \L\'od\'z University of Technology, there is a research group studying Sobolev's legacy (comp. \cite{BeldzinskiGalewski1, BeldzinskiGalewski2, BeldzinskiGalewski3}).

The current paper is was primarily inspired by the works of Przemys\l{}aw G\'orka, Tomasz Kostrzewa and Enrique Reyes, particularly \cite{GorkaKostrzewaReyes}, \cite{GorkaKostrzewaReyes2} and \cite{GorkaReyes}. In their articles, they defined and studied Sobolev spaces on locally compact \textit{abelian} groups. It turned out that much of the classical theory, when the domain is $[0,1]$ or $\reals$, can be recovered in the setting of locally compact abelian groups. It is of crucial importance that the tools that G\'orka, Kostrzewa and Reyes used in their work come from the field of abstract harmonic analysis. This is vital, because as a general rule of thumb ``harmonic analysis on locally compact abelian groups can be generalized to harmonic analysis on Gelfand pairs'' (naturally there are notable exceptions to that ``rule''). This is our starting point and, at the same time, one of the main ideas permeating the whole paper.

Chapter \ref{chapter:HarmonicanalysisGelfandpairs} serves the purpose of laying out the preliminaries of harmonic analysis on Gelfand pairs. Our exposition mainly follows (but is not restricted to) van Dijk's ``Introduction to Harmonic Analysis and Generalized Gelfand pairs'' (comp. \cite{Dijk}) and Wolf's ``Harmonic Analysis on Commutative Spaces'' (comp. \cite{Wolf}). We recap the basic properties of positive-semidefinite and spherical functions and summarize the notion of Gelfand pairs. We proceed with a brief description of the counterpart of Fourier transform in the context of Gelfand pairs, namely the spherical transform. We invoke the celebrated Plancherel theorem, which can be viewed as a kind of bridge between the functions on group $G$ and functions on the dual object of $G$. A significant amount of space is devoted to Hausdorff-Young inequality (which follows from Riesz-Thorin interpolation) and its inverse - we even go to some lengths to provide a not-so-well-known proof of the latter.

Chapter \ref{chapter:Sobolevspaces} aspires to be the main part of the paper. It focuses on the investigation of Sobolev spaces on Gelfand pairs. Theorems \ref{embedding}, \ref{embeddingcontinuous} and \ref{Sobolevembeddingtheorem} generalize the acclaimed Sobolev embedding theorems. It is safe to say that Theorem \ref{RellichKondrachov} is the climax of the paper. It is the counterpart of the eminent Rellich-Kondrachov theorem from the classical theory of Sobolev spaces. 

\section{Harmonic analysis on Gelfand pairs}
\label{chapter:HarmonicanalysisGelfandpairs}

The current chapter serves as a rather brief overview of the harmonic analysis on Gelfand pairs. We commence with the concept of positive-semidefinite functions, explaining their basic properties in \mbox{Lemma \ref{propertiesofpositivesemidefinite}}. Subsequently we introduce the notion of spherical functions and define the spherical transform, whose theory bears a striking resemblance to the theory of the Fourier transform for locally compact abelian groups. We discuss a counterpart of the Plancherel theorem (Theorem \ref{Planchereltheorem}) as well as the inverse spherical transform. The closing part of the chapter is centred around the Hausdorff-Young inequality (Theorem \ref{HausdorffYoung}) and its inverse (Theorem \ref{inverseHausdorffYoung}), which are applied in the sequel.

\subsection{Positive-semidefinite functions, spherical functions and Gelfand pairs}

First of all, let us emphasize that from this point onward, we work under the assumption that 
\begin{center}
\textit{$G$ is a locally compact (Hausdorff) group.}
\end{center}

\noindent
Later on, we will add more assumptions on $G$, but it will \textit{never} cease to be at least a locally compact group. Now, without further ado, we introduce the concept of \textit{positive-semidefinite functions} (in the definition below, and throughout the whole paper $\dagger$ stands for the Hermitian transpose):

\begin{defin}(comp. Definition 32.1 in \cite{HewittRoss2}, p. 253 or Definition 8.4.1 in \cite{Wolf}, p. 165)\\
A function $\phi\colon G\longrightarrow\complex$ is called positive-semidefinite if 
\begin{gather}
\left(\begin{array}{c}
z_1\\
z_2\\
\vdots\\
z_N
\end{array}\right)^{\dagger}\cdot
\left(\begin{array}{cccc}
\phi\left(x_1^{-1}x_1\right) & \phi\left(x_1^{-1}x_2\right) & \ldots & \phi\left(x_1^{-1}x_N\right)\\
\phi\left(x_2^{-1}x_1\right) & \phi\left(x_2^{-1}x_2\right) & & \phi\left(x_2^{-1}x_N\right)\\
\vdots & & \ddots & \vdots\\
\phi\left(x_N^{-1}x_1\right) & \phi\left(x_N^{-1}x_2\right) & \ldots & \phi\left(x_N^{-1}x_N\right)\\
\end{array}\right)
\cdot 
\left(\begin{array}{c}
z_1\\
z_2\\
\vdots\\
z_N
\end{array}\right) \geqslant 0,
\label{positivesemidefinitematrix}
\end{gather}

\noindent
for every $N\in\naturals,\ (x_n)_{n=1}^N\subset G$ and $(z_n)_{n=1}^N \subset \complex^N$. 
\end{defin}

For a fixed sequence of elements $(x_n)_{n=1}^N\subset G$, the condition (\ref{positivesemidefinitematrix}) is known as the positive-semidefiniteness of the matrix
\begin{gather}
\left(\begin{array}{cccc}
\phi\left(x_1^{-1}x_1\right) & \phi\left(x_1^{-1}x_2\right) & \ldots & \phi\left(x_1^{-1}x_N\right)\\
\phi\left(x_2^{-1}x_1\right) & \phi\left(x_2^{-1}x_2\right) & & \phi\left(x_2^{-1}x_N\right)\\
\vdots & & \ddots & \vdots\\
\phi\left(x_N^{-1}x_1\right) & \phi\left(x_N^{-1}x_2\right) & \ldots & \phi\left(x_N^{-1}x_N\right)\\
\end{array}\right).
\label{thismatrixissemidefinite}
\end{gather}

\noindent
To reiterate, a function $\phi$ is positive-semidefinite if all such matrices, regardless of the size $N$ and the sequence $(x_n)_{n=1}^N\subset G$, are positive-semidefinite. We can also express condition (\ref{positivesemidefinitematrix}) more explicitly: a function $\phi$ satisfies (\ref{positivesemidefinitematrix}) if for every $N\in\naturals,\ (x_n)_{n=1}^N\subset G$ and $(z_n)_{n=1}^N\subset \complex$ we have
\begin{gather}
\sum_{n=1}^N\sum_{m=1}^N\ \phi\left(x_n^{-1}x_m\right)\cdot \overline{z_n}\cdot z_m \geqslant 0.
\label{positivesemidefiniteequivalent}
\end{gather}

It will shortly become apparent that positive-semidefinite functions play a crucial role in our understanding of harmonic analysis on Gelfand pairs. We should therefore become acquainted with the properties of these functions. To this end we invoke the following result:

\begin{lemma}(comp. Lemma 5.1.8 in \cite{Dijk}, p. 54, Theorem 32.4 in \cite{HewittRoss2} or Proposition 8.4.2 in \cite{Wolf}, p. 165)\\
If $\phi\colon G\longrightarrow \complex$ is a positive-semidefinite function, then
\begin{enumerate}
	\item $\phi(e)$ is a real and nonnegative number,
	\item $\phi\left(x^{-1}\right) = \overline{\phi(x)}$ for every $x\in G$,
	\item $|\phi(x)| \leqslant \phi(e)$ for every $x\in G$,
	\item $\phi$ is continuous.
\end{enumerate}
\label{propertiesofpositivesemidefinite}
\end{lemma}

To make further progress we take a compact subgroup $K$ of the locally compact (Hausdorff) group $G$ (we shall repeat this assumption ad nauseam) and consider the \textit{double coset space} (comp. \cite{Bechtell}, \mbox{p. 101}):
$$\KGK := \big\{KgK\ :\ g\in G\big\}.$$

\noindent
Let $\pi\colon G\longrightarrow \KGK$ be the projection defined by $\pi(g) := KgK.$ If 
\begin{itemize}
	\item $C_c(G)$ is the space of all continuous functions on $G$ with compact support, and
	\item $C_c(\KGK)$ is the space of all continuous functions on $\KGK$ with compact support
\end{itemize}

\noindent
then every function $F\in C_c(\KGK)$ determines a function $F\circ \pi\colon G\longrightarrow \complex$ belonging to the set of all \textit{bi$-K-$invariant} continuous functions on $G$ with compact support:
\begin{gather}
\bigg\{f\in C_c(G)\ :\ \forall_{\substack{k_1,k_2\in K\\ x\in G}}\ f(k_1xk_2) = f(x)\bigg\}.
\label{biKinvariantfunctions}
\end{gather}

\noindent
Furthermore, the map $F\mapsto F\circ \pi$ is a bijection, and thus we identify $C_c(\KGK)$ with the set (\ref{biKinvariantfunctions}). An analogous reasoning works for the space $C(\KGK)$ of continuous functions, the space $C_0(\KGK)$ of continuous functions which vanish at infinity and the space $L^1(\KGK)$ of integrable functions. If it turns out that $L^1(\KGK)$ is a commutative convolution algebra, the pair $(G,K)$ is called a \textit{Gelfand pair}.

The simplest instance of a Gelfand pair is $(G,\{e\})$ where $G$ is a locally compact \textit{abelian} group. A less trivial example is $(E(n),SO(n)),$ where $SO(n)$ is the special orthogonal group and $E(n)$ is the group of Euclidean motions on $\reals^n$ ($E(n)$ is in fact a semi-direct product of the translation group $T(n)$ and the orthogonal group $O(n)$). Other instances of Gelfand pairs include:
\begin{itemize}
	\item $(GL(n,\reals),O(n)),$ where $GL(n,\reals)$ is the general linear group on $\reals^n,$
	\item $(GL(n,\complex),U(n)),$ where $U(n)$ is the unitary group,
	\item $(O(n+k),O(n)\times O(k)),$
	\item $(SO(n+k),SO(n)\times SO(k)),$
	\item $(U(n+k),U(n)\times U(k)),$
	\item $(SU(n+k),SU(n)\times SU(k)),$ where $SU(n)$ is the special unitary group.
\end{itemize}

Given a Gelfand pair $(G,K)$, every function $\chi\colon C_c(\KGK)\longrightarrow \complex$ satisfying 
$$\forall_{f,g\in C_c(\KGK)}\ \chi(f\star g) = \chi(f)\cdot \chi(g),$$

\noindent
where $\star$ denotes the convolution in $C_c(\KGK),$ is called a \textit{character}. Furthermore, every continuous, bi$-K-$invariant function $\phi\colon G\longrightarrow \complex$ for which 
$$\forall_{f\in C_c(\KGK)}\ \chi_{\phi}(f) := \int_G\ f(x)\cdot \phi(x)\ dx$$

\noindent
is a nontrivial character, is called a \textit{spherical function}. These functions admit the following characterization:

\begin{thm}(comp. Propositions 6.1.5 and 6.1.6 in \cite{Dijk}, p. 77, Proposition 2.2 in \cite{Helgason}, p. 400 or Theorem 8.2.6 in \cite{Wolf}, p. 157)\\
The following conditions are equivalent:
\begin{itemize}
	\item $\phi\colon G\longrightarrow \complex$ is a spherical function.
	\item $\phi\colon G\longrightarrow \complex$ is a nonzero, continuous function such that 
	$$\forall_{x,y\in G}\ \int_K\ \phi(xky)\ dk =\phi(x)\cdot \phi(y),$$

	\noindent
	where $dk$ is the normalized Haar measure on $K$. 
	\item $\phi\colon G\longrightarrow \complex$ is a continuous, bi$-K-$invariant function with $\phi(e) = 1$ and for every $f\in C_c(\KGK)$ there exists a complex number 
	$$\lambda_f = \int_G\ f(x)\cdot \phi\left(x^{-1}\right)\ dx$$
	
	\noindent
	such that $f\star \phi = \lambda_f\cdot\phi.$
\end{itemize}
\label{characterizationofsphericalfunctions}
\end{thm}

\subsection{Spherical transform}

The previous section concluded with a characterization of spherical functions. Currently, our objective is to employ these functions to describe the spherical transform (and its inverse), which is a counterpart of the prominent Fourier transform on locally compact abelian groups. To this end, we introduce the following notation:
\begin{itemize}
	\item $\Spher(G,K)$ denotes the set of all spherical functions on the Gelfand pair $(G,K),$
	\item $\Spher^b(G,K)$ denotes the set of all \textit{bounded} spherical functions on $(G,K),$ 
	\item $\Spher^+(G,K)$ denotes the set of all \textit{positive-semidefinite} spherical functions on $(G,K).$ 
\end{itemize}

\noindent
Let us remark that due to Lemma \ref{propertiesofpositivesemidefinite}, positive-semidefinite functions are automatically bounded, so $\Spher^+(G,K)\subset \Spher^b(G,K).$ Furthermore, since the Gelfand pair $(G,K)$ is usually understood from the context, we frequently simplify the notation $\Spher(G,K),\ \Spher^b(G,K)$ and $\Spher^+(G,K)$ to $\Spher,\ \Spher^b$ and $\Spher^+,$ respectively. 

\begin{defin}(comp. Definition 6.4.3 in \cite{Dijk}, p. 83 or Definition 9.2.1 in \cite{Wolf}, p. 184)\\
For every $f\in L^1(\KGK)$ the function $\widehat{f}\colon \Spher^b \longrightarrow \complex$ given by
$$\widehat{f}(\phi) := \int_G f(x)\cdot \phi\left(x^{-1}\right)\ dx$$

\noindent
is called a spherical transform. 
\end{defin}

As in \cite{Wolf}, p. 185 we topologize $\Spher^b$ (which is a maximal ideal space for the commutative Banach algebra $L^1(\KGK)$) with the weak topology from the family of maps $\big\{\widehat{f}\ :\ f\in L^1(\KGK)\big\}$ - this turns $\Spher^b$ into a locally compact Hausdorff space. In the sequel we will focus on the subspace $\Spher^+$ of $\Spher^b$, with the induced topology. It is remarkable that the induced weak topology on $\Spher^+$ coincides with the topology of uniform convergence on compact subsets of $G$ (comp. Proposition 6.4.2 in \cite{Dijk}, p. 83 or Theorem 3.31 in \cite{Folland}, p. 82).  

Let us define the function $\Gamma(f) := \widehat{f}$ for every $f\in L^1(\KGK).$ Since every spherical transform $\widehat{f}$ is a continuous function which vanishes at infinity (comp. Corollary 9.1.14 or Corollary 9.2.10 in \cite{Wolf}), we have
$$\Gamma\colon L^1(\KGK) \longrightarrow C_0(\Spher^b).$$

\noindent
This is a counterpart of the classical Riemann-Lebesgue lemma, a classical result in the theory of Fourier transform (comp. Lemma 3.3.7 in \cite{Deitmar}, p. 47 or Theorem 1.7 in \cite{Katznelson}, p. 136). Furthermore, we define 
$$\X_1 := B(\KGK)\cap L^1(\KGK),$$

\noindent
where $B(\KGK)$ is the set of all linear combinations of positive-semidefinite and bi$-K-$invariant functions. Theorem 9.4.1 in \cite{Wolf}, p. 191 asserts the existence of a measure $\widehat{\mu}$ on $\Spher^+$ such that for every $f\in \X$ we have $\widehat{f}\in L^1(\Spher^+,\widehat{\mu})$ and 
\begin{gather}
\forall_{x\in G}\ f(x) = \int_{\Spher^+} \widehat{f}(\phi)\cdot \phi(x)\ d\widehat{\mu}(\phi).
\label{inversesphericaltransform}
\end{gather}

\noindent
In other words $\Gamma$ is a bijection between $\X\subset L^1(\KGK)$ and the image $\Gamma(\X)\subset L^1(\Spher^+,\widehat{\mu}).$ This insight is vital while proving the counterpart of the Plancherel formula:

\begin{thm}(comp. Theorem 6.4.6 in \cite{Dijk}, p. 85 or Theorem 9.5.1 in \cite{Wolf}, p. 193)\\
If $f\in L^1(\KGK) \cap L^2(\KGK)$ then $\widehat{f}\in L^2(\Spher^+,\widehat{\mu})$ and $\|f\|_2 = \|\widehat{f}\|_2.$
\label{Planchereltheorem}
\end{thm}

As in the classical case (comp. Theorem 3.5.2 in \cite{Deitmar}, p. 53), by Theorem 1.7 in \cite{ReedSimon}, p. 9 we can extend $\Gamma$ (a linear and bounded operator defined on a dense set $L^1(\KGK)\cap L^2(\KGK)$ in $L^2(\KGK)$) to an isometric isomorphism between $L^2(\KGK)$ and $L^2(\Spher^+,\widehat{\mu}).$ In order to simplify the notation, we still denote this extension by $\Gamma.$ The map $\Gamma^{-1}$ is called the \textit{inverse spherical transform} and we often write $\widecheck{F} := \Gamma^{-1}(F)$ for $F\in L^2(\Spher^+,\widehat{\mu}).$

As a corollary to Plancherel theorem (Theorem \ref{Planchereltheorem}) we have the following result:

\begin{cor}(Corollary 9.5.2 in \cite{Wolf}, p. 194)\\
If $f,g \in L^2(\KGK)$, then
$$\langle f|g \rangle_{L^2(\KGK)} = \langle \widehat{f}|\widehat{g} \rangle_{L^2(\Spher^+,\widehat{\mu})}.$$
\label{corollaryofPlancherel}
\end{cor}

In the sequel we abbreviate both inner products $\langle \cdot|\cdot \rangle_{L^2(\KGK)}$ and $\langle \cdot|\cdot \rangle_{L^2(\Spher^+,\widehat{\mu})}$ to $\langle \cdot|\cdot \rangle_2$. The distinction between the two should be obvious from the context.

\subsection{Hausdorff-Young inequality and its inverse}

The final section of the current chapter opens with a renowned \textit{Riesz-Thorin interpolation theorem}:

\begin{thm}(comp. Theorem 6.27 in \cite{FollandRealAnalysis}, p. 200)\\
Let $(X,\Sigma_X,\mu_X)$ and $(Y,\Sigma_Y,\mu_Y)$ be measure spaces and let $p_0,p_1,q_0,q_1\in [1,\infty]$. If $q_0=q_1 = \infty,$ suppose also that $\mu_Y$ is semifinite. For $0 < t < 1$, define $p_t$ and $q_t$ by
\begin{gather*}
\frac{1}{p_t} := \frac{1-t}{p_0} + \frac{t}{p_1} \hspace{0.4cm}\text{and}\hspace{0.4cm} \frac{1}{q_t} := \frac{1-t}{q_0} + \frac{t}{q_1}.
\end{gather*}

\noindent
If $T$ is a linear map from $L^{p_0}(X,\mu_X) + L^{p_1}(X,\mu_X)$ to $L^{q_0}(X,\mu_Y) + L^{q_1}(X,\mu_Y)$ such that there exist constants $M_0, M_1 >0$ with
\begin{equation*}
\begin{split}
\|Tf\|_{q_0} &\leqslant M_0\cdot \|f\|_{p_0} \hspace{0.4cm}\text{for every } f\in L^{p_0}(X,\mu_X),\text{ and}\\
\|Tf\|_{q_1} &\leqslant M_1\cdot \|f\|_{p_1} \hspace{0.4cm}\text{for every } f\in L^{p_1}(X,\mu_X),
\end{split}
\end{equation*}

\noindent
then 
\begin{gather*}
\|Tf\|_{q_t} \leqslant M_0^{1-t}M_1^t\cdot \|f\|_{p_t}\hspace{0.4cm}\text{for every } f\in L^{p_t}(X,\mu_X),\ t\in (0,1).
\end{gather*}
\label{RieszThorin}
\end{thm}

Riesz-Thorin interpolation theorem is a rather advanced tool and one of its corollaries is the aforementioned \textit{Hausdorff-Young inequality}:

\begin{thm}(comp. Theorem 31.20 in \cite{HewittRoss2}, p. 226 or \cite{Wolf}, p. 200)\\
If $p\in [1,2]$ then for every $f\in L^p(\KGK)$ the inequality $\|\widehat{f}\|_{p'}\leqslant \|f\|_p$ holds ($p'$ stands for the H\"{o}lder conjugate of $p$). 
\label{HausdorffYoung}
\end{thm}

In the sequel, we will also need the \textit{inverse Hausdorff-Young inequality}, which we present below. We take the liberty of providing a rather short proof as it is far less known than the proof of the ``classical'' Hausdorff-Young inequality.

\begin{thm}(comp. Theorem 31.24 in \cite{HewittRoss2}, p. 229 or \cite{Wolf}, p. 200)\\
If $p\in (1,2]$ then for every $f\in L^2(\KGK)$ the inequality $\|f\|_{p'}\leqslant \|\widehat{f}\|_p$ holds. 
\label{inverseHausdorffYoung}
\end{thm}
\begin{proof}
Firstly, if $\|\widehat{f}\|_p = \infty$ then we are immediately done. Thus without loss of generality, we suppose that $\widehat{f} \in L^p(\Spher^+,\widehat{\mu}).$

We start off with an observation that for every function $g \in L^2(\KGK)\cap L^p(\KGK)$ we have
\begin{gather*}
\left|\int_G\ g(x)\cdot\overline{f(x)}\ dx\right| = \left|\langle g|f \rangle_2\right| \stackrel{\text{Corollary } \ref{corollaryofPlancherel}}{=} \left|\langle \widehat{g}|\widehat{f} \rangle_2\right| \stackrel{\text{H\"older ineq.}}{\leqslant} \|\widehat{g}\|_{p'}\cdot \|\widehat{f}\|_p \stackrel{\text{Theorem } \ref{HausdorffYoung}}{\leqslant} \|g\|_p \cdot \|\widehat{f}\|_p.
\end{gather*}  

\noindent
Consequently, the map $g\mapsto \int_G\ g(x)\cdot \overline{f(x)}\ dx$ is a linear and bounded functional on $L^2(\KGK)\cap L^p(\KGK)$. Since $L^2(\KGK)\cap L^p(\KGK)$ is dense in $L^p(\KGK)$ we can extend this map to a linear and bounded functional $\Phi\colon L^p(\KGK)\longrightarrow \complex$ such that $\|\Phi\|\leqslant \|\widehat{f}\|_p.$ By Riesz representation theorem (comp. \cite{Brezis}, Theorem 4.11, p. 97) there exists a function $h\in L^{p'}(\KGK)$ such that 
$$\forall_{g\in L^p(\KGK)}\ \Phi(g) = \int_G\ g(x)\cdot h(x)\ dx \hspace{0.4cm}\text{and}\hspace{0.4cm} \|h\|_{p'} = \|\Phi\|.$$

\noindent
It follows that $\overline{f}\equiv h$ almost everywhere. Lastly, we have
$$\|f\|_{p'} = \|h\|_{p'} = \|\Phi\| \leqslant \|\widehat{f}\|_p,$$

\noindent
which concludes the proof.
\end{proof}

\section{Sobolev spaces}
\label{chapter:Sobolevspaces}

After the revision of harmonic analysis on Gelfand pairs in the previous chapter, we are in position to define Sobolev spaces on Gelfand pair $(G,K)$. Although Sobolev spaces are frequently introduced via the notion of a weak derivative, the expression ``differentiation on a group'' may be meaningless, so we must resort to a different approach than suggested in \cite{Brezis}, p. 202 or \cite{RenardyRogers}, p. 204 (to be fair, the latter discusses the characterization of Sobolev spaces via the Fourier transform a couple of pages further). We pursue a path inspired by the works of G\'{o}rka, Kostrzewa, Reyes (comp. \cite{GorkaKostrzewaReyes, GorkaKostrzewaReyes2, GorkaReyes}) as well as Grafakos (comp. \cite{Grafakos}, p. 13), Malliavin (comp. \cite{Malliavin}, p. 142) or Triebel (comp. \cite{Triebel}, p. 3).

\begin{defin}
For a measurable function $\gamma\colon \Spher^+\longrightarrow \reals_+$ and $s\in \reals_+$, the set 
\begin{gather}
H^s_{\gamma}(\KGK) := \bigg\{f\in L^2(\KGK)\ :\ \int_{\Spher^+}\ \left(1+\gamma(\phi)^2\right)^s\cdot |\widehat{f}(\phi)|^2\ d\widehat{\mu}(\phi) < \infty\bigg\}
\label{sobspace}
\end{gather}

\noindent
is called a Sobolev space. 
\label{Sobolevspacedefinition}
\end{defin}

Let us immediately explain the reason why the above definition encompasses the definitions encountered in the literature: 
\begin{itemize}
	\item A Sobolev space (with weight $\gamma$) on $\reals$ can be defined as 
\begin{gather}
H^s_{\gamma}(\reals) = \bigg\{f\in L^2(\reals)\ :\ \int_{\reals}\ \left(1+\gamma(y)^2\right)^s\cdot |\widehat{f}(y)|^2\ dy < \infty\bigg\}.
\label{sobolevonreals}
\end{gather}

\noindent
If $\gamma(y)= |y|^p,\ p\in\reals_+$ and $s=1$, then (\ref{sobolevonreals}) is exactly the Definition 7.14 on page 208 in \cite{RenardyRogers}. A similar approach is presented in \cite{Taylor} on page 271. This is perfectly compatible with Definition \ref{Sobolevspacedefinition}: if $G = \reals$ and $K = \{0\}$, then (comp. \cite{Dijk}, p. 77) we have $\Spher = \left\{x\mapsto e^{\lambda x}\ :\ \lambda \in\complex\right\}.$ Furthermore, if $\phi(x) = e^{\lambda x}$ is an element of $\Spher^+$ then (\ref{positivesemidefiniteequivalent}) implies (for $N=2,\ z_1 = 1,\ z_2 = i$) that
$$2 + i\left(e^{\lambda(x_2-x_1)} - e^{-\lambda (x_2-x_1)}\right) \geqslant 0.$$

\noindent
This is possible only if $\text{Re}(\lambda) = 0,$ which means that $\Spher^+ \subset \left\{x\mapsto e^{iyx}\ :\ y\in\reals\right\}$. Moreover, for every $y\in\reals$ and $(x_n)_{n=1}^N\subset \reals,\ (z_n)_{n=1}^N\subset \complex$ we have
\begin{gather*}
\sum_{n=1}^N\sum_{m=1}^N\ e^{iy(x_m-x_n)}\cdot\overline{z_n}\cdot z_m = \left(\sum_{n=1}^N\ e^{iyx_n}\cdot z_n\right)\cdot \overline{\left(\sum_{n=1}^N\ e^{iyx_n}\cdot z_n\right)} \geqslant 0,
\end{gather*}

\noindent
so we conclude that $\Spher^+ = \left\{x\mapsto e^{iyx}\ :\ y\in\reals\right\}.$ In other words, $\Spher^+ = \widehat{\reals} \cong \reals$ (comp. \cite{Deitmar}, p. 101-102) so (\ref{sobspace}) infallibly reconstructs (\ref{sobolevonreals}).

\item The reasoning from the first case can be generalized as follows: if $G$ is a locally compact abelian group and $K =\{e\}$, then by Theorem 5.3.3, p. 61 and Theorem 6.2.5, p. 81 in \cite{Dijk} we conclude that $\Spher^+ = \widehat{G}.$ Consequently, (\ref{sobspace}) reads
$$H^s_{\gamma}(G) = \bigg\{f\in L^2(G)\ :\ \int_{\widehat{G}}\ \left(1+\gamma(y)^2\right)^s\cdot |\widehat{f}(y)|^2\ d\widehat{\mu}(y) < \infty\bigg\},$$

\noindent
which is precisely the definition of G\'orka and Kostrzewa (comp. \cite{GorkaKostrzewaReyes, GorkaKostrzewaReyes2, GorkaReyes}). 
\end{itemize}

Having justified the way we chose to define the Sobolev spaces on a Gelfand pair $(G,K)$ we take a moment to dwell on the space $H^s_{\gamma}(\KGK).$ It is obviously a linear subspace of $L^2(\KGK)$ and for every function $f\in H^s_{\gamma}(\KGK)$ we can define 
$$\|f\|_{H^s_{\gamma}} := \left(\int_{\Spher^+}\ \left(1+\gamma(\phi)^2\right)^s\cdot |\widehat{f}(\phi)|^2\ d\nu(\phi)\right)^{\frac{1}{2}}.$$

\noindent
We briefly argue that it is in fact a norm, which turns $H^s_{\gamma}(\KGK)$ into a Banach space. 

Obviously, $\|0\|_{H^s_{\gamma}} = 0$ so let $f \in H^s_{\gamma}(\KGK)$ be such that $\|f\|_{H^s_{\gamma}} = 0$. This implies that $|\widehat{f}(\phi)| = 0$ for every $\phi\in \Spher^+.$ By Theorem \ref{Planchereltheorem} we conclude that $\|f\|_2 = 0,$ which means $f = 0$ almost everywhere. 

It is trivial to see that $\|\alpha f\|_{H^s_{\gamma}} = |\alpha|\cdot \|f\|_{H^s_{\gamma}}$ for every $f\in H^s_{\gamma}(\KGK)$ and $\alpha\in\complex$. Lastly, $\|\cdot\|_{H^s_{\gamma}}$ obeys the triangle inequality due to the classical Minkowski inequality (comp. Theorem 6.5 in \cite{FollandRealAnalysis}, p. 183). In summary, $\|\cdot\|_{H^s_{\gamma}}$ is a norm in $H^s_{\gamma}(\KGK)$. 

As far as the completeness of $(H^s_{\gamma}(\KGK),\|\cdot\|_{H^s_{\gamma}})$ is concerned, we could try to prove it in a tedious manner ``from scratch''. However, in order to save time let us resort to a cunning trick which makes the problem significantly easier: by Theorem \ref{Planchereltheorem} we know that the space $H^s_{\gamma}(\KGK)$ is isometrically isomorphic (via the spherical transform) to the space $L^2(\Spher^+,\nu)$ of square-integrable functions on $\Spher^+$ with respect to the measure $d\nu = (1+\gamma^2)^{\frac{s}{2}} d\widehat{\mu}.$ Since $L^2(\Spher^+,\nu)$ is complete, then so is $H^s_{\gamma}(\KGK)$. We conclude that $H^s_{\gamma}(\KGK)$ is a Banach space.

\subsection{Sobolev embedding theorems}
\label{chapter:sobolevembeddingtheorems}

Our next big topic is the embedding theorems. In general, we say (comp. Definition 7.15 in \cite{RenardyRogers}, p. 209) that a Banach space $X$ is continuously embedded in a Banach space $Y$ if $X\subset Y$ and there exists a constant $C>0$ such that 
$$\forall_{u\in X}\ \|u\|_Y \leqslant C\cdot\|u\|_X.$$

\noindent
We denote this situation by $X\hookrightarrow Y$. 

The question that permeates the current section is: ``In what Banach spaces can be embed the Sobolev spaces $H^s_{\gamma}(\KGK)$ defined previously?'' This is by far not a trivial task, as is confirmed by Chapter 4 in \cite{AdamsFournier}, Theorem 8.8 in \cite{Brezis}, p. 212, Chapter 2 in \cite{Demengels}, Chapter 5.6 in \cite{Evans}, Chapters 7.2.3 and 7.2.4 in \cite{RenardyRogers} or Chapter 2.4 in \cite{Ziemer}. 

\begin{thm}
For every $f\in H^s_{\gamma}(\KGK)$ we have $\|f\|_2 \leqslant \|f\|_{H^s_{\gamma}}.$
\label{embedding}
\end{thm}
\begin{proof}
For every function $f\in H^s_{\gamma}(\KGK)$ we have
\begin{gather*}
\|f\|_2 \stackrel{\text{Theorem}\ \ref{Planchereltheorem}}{=} \|\widehat{f}\|_2 = \left(\int_{\Spher^+}\ |\widehat{f}(\phi)|^2\ d\widehat{\mu}(\phi)\right)^{\frac{1}{2}} \leqslant \left(\int_{\Spher^+}\ \left(1+\gamma(\phi)^2\right)^s\cdot |\widehat{f}(\phi)|^2\ d\widehat{\mu}(\phi)\right)^{\frac{1}{2}} = \|f\|_{H^s_{\gamma}},
\end{gather*}

\noindent
which concludes the proof.
\end{proof}

To put it differently, Theorem \ref{embedding} says that the Sobolev space $H^s_{\gamma}(\KGK)$ is \textit{continuously embedded} in $L^2(\KGK),$ i.e. $H^s_{\gamma}(\KGK)\hookrightarrow L^2(\KGK).$ Our next results embeds (under certain circumstances) the Sobolev space $H^s_{\gamma}(\KGK)$ in the space of continuous functions: 

\begin{thm}
If 
$$\left(1+\gamma^2\right)^{-\frac{s}{2}} \in L^2(\Spher^+,\widehat{\mu})$$

\noindent
then $H^s_{\gamma}(\KGK)\hookrightarrow C(\KGK),$ i.e. $H^s_{\gamma}(\KGK)$ embeds continuously in $C(\KGK).$
\label{embeddingcontinuous}
\end{thm}
\begin{proof}
Fix $x_*\in G$ and $\eps > 0$. Since $\Spher^+$ is equipped with the topology of uniform convergence on compact sets, then (by the Arzel\`a-Ascoli theorem) $\Spher^+$ is equicontinuous: there exists an open neighbourhood $U$ of $x_*$ such that 
\begin{gather}
\forall_{\substack{x\in U\\ \phi\in \Spher^+}}\ |\phi(x) - \phi(x_*)| < \eps.
\label{equicontinuitySpher}
\end{gather}

\noindent
Consequently, for every function $f\in H^s_{\gamma}(\KGK)$ and $x\in U$ we have
\begin{gather*}
|f(x) - f(x_*)| = \left|\int_{\Spher^+}\ \widehat{f}(\phi)\cdot \big(\phi(x) - \phi(x_*)\big)\ d\widehat{\mu}(\phi)\right| \leqslant \sup_{\phi\in \Spher^+}\ |\phi(x)-\phi(x_*)|\cdot \int_{\Spher^+}\ |\widehat{f}(\phi)|\ d\widehat{\mu}(\phi) \\
\stackrel{\text{Cauchy-Schwarz ineq.}}{\leqslant} \sup_{\phi\in \Spher^+}\ |\phi(x) - \phi(x_*)|\cdot \left(\int_{\Spher^+}\ \left(1+\gamma(\phi)^2\right)^s\cdot |\widehat{f}(\phi)|^2\ d\widehat{\mu}(\phi)\right)^{\frac{1}{2}} \cdot \left(\int_{\Spher^+}\ \left(1+\gamma(\phi)^2\right)^{-s}\ d\widehat{\mu}(\phi)\right)^{\frac{1}{2}} \\
\leqslant \sup_{\phi\in \Spher^+}\ |\phi(x) - \phi(x_*)|\cdot \|f\|_{H^s_{\gamma}}\cdot \|\left(1+\gamma^2
\right)^{-\frac{s}{2}}\|_2 \stackrel{(\ref{equicontinuitySpher})}{\leqslant} \eps \cdot \|f\|_{H^s_{\gamma}}\cdot \|\left(1+\gamma^2
\right)^{-\frac{s}{2}}\|_2,
\end{gather*}

\noindent
which proves that $H^s_{\gamma}(\KGK)\subset C(\KGK).$ Finally, a similar estimate to the one above leads to 
\begin{gather*}
|f(x)| \leqslant \sup_{\phi\in \Spher^+}\ |\phi(x)|\cdot \|f\|_{H^s_{\gamma}}\cdot \|\left(1+\gamma^2
\right)^{-\frac{s}{2}}\|_2 \stackrel{\text{Lemma}\ \ref{propertiesofpositivesemidefinite}}{\leqslant} \sup_{\phi\in \Spher^+}\ \phi(e)\cdot \|f\|_{H^s_{\gamma}}\cdot \|\left(1+\gamma^2\right)^{-\frac{s}{2}}\|_2 \leqslant \|f\|_{H^s_{\gamma}}\cdot \|\left(1+\gamma^2\right)^{-\frac{s}{2}}\|_2, 
\end{gather*}

\noindent
where the last inequality stems from the fact that $\phi(e) = 1$ for every spherical function $\phi.$ This concludes the proof.
\end{proof}

Finally, we prove that (under certain assumptions) Sobolev space $H^s_{\gamma}(\KGK)$ embeds continuously in $L^{p'}(\KGK).$ Let us recall that $p'$ is the H\"{o}lder conjugate of $p$.

\begin{thm}
Let $\alpha > s >0$ and let $p := \frac{2\alpha}{\alpha + s}.$ If 
$$\left(1+\gamma^2\right)^{-1}\in L^{\alpha}(\Spher^+,d\widehat{\mu}),$$

\noindent
then $H^s_{\gamma}(\KGK)\hookrightarrow L^{p'}(\KGK).$
\label{Sobolevembeddingtheorem}
\end{thm}
\begin{proof}
Since $p\in (1,2]$ then by Theorem \ref{inverseHausdorffYoung} we know that for every $f\in L^2(\KGK)$ we have $\|f\|_{p'} \leqslant \|\widehat{f}\|_p$. Furthermore, we have
\begin{gather*}
\|\widehat{f}\|_p = \left(\int_G\ |\widehat{f}(\phi)|^p\cdot \frac{(1+\gamma(\phi)^2)^{\frac{sp}{2}}}{(1+\gamma(\phi)^2)^{\frac{sp}{2}}}\ d\widehat{\mu}(\phi)\right)^{\frac{1}{p}} \stackrel{\text{H\"{o}lder ineq.}}{\leqslant} \|f\|_{H^s_{\gamma}}\cdot \left(\int_{\Spher^+}\ (1+\gamma(\phi)^2)^{-\frac{sp}{2-p}}\ d\widehat{\mu}(\phi)\right)^{\frac{2-p}{2p}}.
\end{gather*}

\noindent
Since $\alpha = \frac{sp}{2-p}$ we finally obtain
\begin{gather*}
\|f\|_{p'} \leqslant \|\widehat{f}\|_p \leqslant \|f\|_{H^s_{\gamma}}\cdot \|\left(1+\gamma^2\right)^{-1}\|_{\alpha}^{\frac{s}{2}},
\end{gather*}

\noindent
which concludes the proof.
\end{proof}

\subsection{Rellich-Kondrachov theorem}

Having discussed the embedding theorems in the previous section, we proceed with the next topic, namely the Rellich-Kondrachov theorem. This is one of the central results in the classical theory of Sobolev spaces and appears in virtually any text on the subject: Theorem 6.3 in \cite{AdamsFournier}, p. 168, Theorem 1 in \cite{Evans}, p. 272, Theorem 11.10 in \cite{Leoni}, p. 320, Theorem 6.1 in \cite{Necas}, p. 102 or Theorem 2.5.1 in \cite{Ziemer}, p. 62.

The main theorem in this section, the counterpart of Rellich-Kondrachov theorem for Gelfand pairs, is Theorem \ref{RellichKondrachov}. Before we dig into the proof of this result we demonstrate three auxilary lemmas (Lemma \ref{auxilarylemma1}, \ref{auxilarylemma2} and \ref{fnetaconvergesfeta}). We should also emphasize the fact that Lemma \ref{fnetaconvergesfeta} and (as a consequence) Theorem \ref{RellichKondrachov} require the group $G$ to be compact while previous lemmas work without that assumption. 

\begin{lemma}
Let $f\in H_{\gamma}^s(\KGK).$ If $y\in G$ then
$$\int_G\ |f\left(xy^{-1}\right) - f(x)|^2\ dx \leqslant \left(\sup_{\phi\in\Spher^+}\ \frac{\left|\phi(y) - 1\right|^2}{(1+\gamma(\phi)^2)^s}\right)\cdot \|f\|_{H_{\gamma}^s}^2.$$
\label{auxilarylemma1}
\end{lemma}
\begin{proof}
We fix $y\in G$ and, for a moment, we suppose that $f\in C_c(\KGK)$. If the map $R_y\colon L^2(\KGK)\longrightarrow L^2(\KGK)$ is given by $R_yf(x) := f\left(xy^{-1}\right),$ then
\begin{equation}
\begin{split}
\forall_{\phi\in\Spher^+}\ \widehat{R_yf}(\phi) &= \int_G\ R_yf(x)\cdot \phi\left(x^{-1}\right)\ dx = \int_G\ f\left(xy^{-1}\right)\cdot\phi\left(x^{-1}\right)\ dx \\
&\stackrel{x\mapsto xy}{=} \int_G\ f(x)\cdot \phi\left(y^{-1}x^{-1}\right)\ dx \stackrel{x\mapsto x^{-1}}{=} \int_G\ f(x)\cdot \phi\left(y^{-1}x^{-1}\right)\ dx \\
&= \phi\star f\left(y^{-1}\right) = f\star \phi\left(y^{-1}\right) \stackrel{\text{Theorem } \ref{characterizationofsphericalfunctions}}{=} \widehat{f}(\phi)\cdot \phi\left(y^{-1}\right) .
\end{split}
\label{sphericaltransformofLyf}
\end{equation}

\noindent
By Theorem \ref{Planchereltheorem} we have
\begin{gather*}
\int_G\ |f\left(xy^{-1}\right) - f(x)|^2\ dx = \int_{\Spher^+}\ \left|\widehat{R_yf}(\phi) - \widehat{f}(\phi)\right|^2\ d\widehat{\mu}(\phi) \stackrel{(\ref{sphericaltransformofLyf})}{=} \int_{\Spher^+}\ |\widehat{f}(\phi)|^2\cdot \left|\phi\left(y^{-1}\right) - 1\right|^2\ d\widehat{\mu}(\phi) \\
\stackrel{\text{Lemma } \ref{propertiesofpositivesemidefinite}}{=} \int_{\Spher^+}\ |\widehat{f}(\phi)|^2\cdot \left(1+\gamma(\phi)^2\right)^s\cdot \frac{\left|\phi(y) - 1\right|^2}{\left(1+\gamma(\phi)^2\right)^s}\ d\widehat{\mu}(\phi) \leqslant \left(\sup_{\phi\in\Spher^+}\ \frac{\left|\phi(y) - 1\right|^2}{\left(1+\gamma(\phi)^2\right)^s}\right)\cdot \|f\|_{H_{\gamma}^s}^2.
\end{gather*}

\noindent
To conclude the proof it suffices to note that $C_c(\KGK)$ is dense in $H_{\gamma}^s(\KGK).$
\end{proof}

Prior to Lemma \ref{auxilarylemma2} we revisit a very popular Minkowski's integral inequality:

\begin{thm}(comp. Theorem 6.19 in \cite{FollandRealAnalysis}, p. 194 or \mbox{\cite{SteinSingular}, p. 271})\\
Let $X,Y$ be $\sigma-$finite measure spaces, $1\leqslant p < \infty$ and let $F:X\times Y \longrightarrow \complex$ be a measurable function. Then
\begin{gather}
\left(\int_X \left(\int_Y\ |F(x,y)|\ dy\right)^p\ dx\right)^{\frac{1}{p}} \leqslant \int_Y \left(\int_X\ |F(x,y)|^p\ dx\right)^{\frac{1}{p}}\ dy.
\label{Minkineq}
\end{gather}
\end{thm}

In the lemma below, $\wedge$ stands for the logical ``and'', while $\vee$ stands for the logical ``or''.

\begin{lemma}
Let $f\in H_{\gamma}^s(\KGK)$. If $\eta \in C_c(\KGK)$ then
\begin{gather*}
\|f\star\eta - f\|_2 \leqslant \sup_{y\in\supp(\eta)}\ \left(\sup_{\phi\in\Spher^+}\ \frac{\left|\phi(y) - 1\right|}{(1+\gamma(\phi)^2)^\frac{s}{2}}\right)\cdot \|f\|_{H_{\gamma}^s}.
\end{gather*}
\label{auxilarylemma2}
\end{lemma}
\begin{proof}
At first, we choose $f_B$ to be a Borel-measurable function such that $f = f_B$ almost everywhere. Since $G$ is a Tychonoff space (as a locally compact group), then there exists $\eta\in C_c(\KGK)$ such that 
\begin{itemize}
	\item $\eta(e)\neq 0,\ \eta \geqslant 0,$ and
	\item $\int_G\ \eta(y)\ dy = 1.$
\end{itemize}

\noindent
For every $x\in G$ we have 
\begin{gather*}
|f\star\eta(x)-f(x)|^2 = \left|\int_G\ f\left(xy^{-1}\right)\cdot \eta(y)\ dy - f(x)\cdot \int_G\ \eta(y)\ dy\right|^2 = \left|\int_G\ \big(f\left(xy^{-1}\right) - f(x)\big)\cdot \eta(y)\ dy\right|^2.
\end{gather*}

We define 
$$F(x,y) := \big(f_B\left(xy^{-1}\right) - f_B(x)\big)\cdot \eta(y),$$ 

\noindent
which is a Borel function as a composition of Borel functions:
$$(x,y)\mapsto (x,y,y)\mapsto (x,y^{-1},y)\mapsto (x,xy^{-1},\eta(y))\mapsto (f_B(x),f_B\left(xy^{-1}\right),\eta(y))\mapsto \big(f_B\left(xy^{-1}\right) - f_B(x)\big)\cdot \eta(y).$$

\noindent 
Since $f_B$ and $\eta$ are both in $L^2(\KGK),$ then $\supp(f_B)$ and $\supp(\eta)$ are $\sigma-$compact (comp. Corollary 1.3.5 in \cite{Deitmar}, p. 10). Consequently, also the sets 
$$\left(\supp(f_B) \cdot\supp(\eta)\right)\times \supp(\eta) \hspace{0.4cm}\text{and}\hspace{0.4cm} \supp(f_B) \times \supp(\eta)$$

\noindent
are $\sigma-$compact. Next, we follow a series of logical implications:
\begin{gather*}
(x,y)\in \{F\neq 0\} \ \Longrightarrow \ \bigg(f_B\left(xy^{-1}\right)-f_B(x)\neq 0 \hspace{0.4cm}\wedge\hspace{0.4cm} \eta(y) \neq 0\bigg)\\
\Longrightarrow\ \bigg(\left(xy^{-1}\in\supp(f_B) \hspace{0.4cm}\vee \hspace{0.4cm} x\in\supp(f_B) \right)\hspace{0.4cm}\wedge \hspace{0.4cm} y \in\supp(\eta) \bigg)\\
\Longrightarrow\ \bigg(\left(xy^{-1}\in\supp(f_B) \hspace{0.4cm}\wedge \hspace{0.4cm} y\in\supp(\eta)\right) \hspace{0.4cm}\vee \hspace{0.4cm} \left(x\in\supp(f_B) \hspace{0.4cm}\wedge \hspace{0.4cm} y\in\supp(\eta)\right) \bigg)\\
\Longrightarrow\ \bigg((x,y)\in \left(\supp(f_B) \cdot\supp(\eta)\right)\times \supp(\eta) \hspace{0.4cm}\vee \hspace{0.4cm} (x,y)\in\supp(f_B)\times \supp(\eta)\bigg).
\end{gather*}

\noindent
We conclude that $\{F\neq 0\}$ is $\sigma-$compact. 

Finally, we are in position to apply Minkowski's integral inequality:
\begin{equation*}
\begin{split}
\|f\star\eta - f\|_2 &= \|f_B\star\eta - f_B\|_2 = \left(\int_G\ \bigg|\int_G\ \big(f_B\left(xy^{-1}\right)-f_B(x)\big)\cdot \eta(y)\ dy\bigg|^2\ dx\right)^{\frac{1}{2}}\\ 
&\stackrel{(\ref{Minkineq})}{\leqslant} \int_G\ \left( \int_G\ |f_B\left(xy^{-1}\right)-f_B(x)|^2\cdot |\eta(y)|^2\ dx\right)^{\frac{1}{2}} \ dy \\
&= \int_G\ \|R_yf_B-f_B\|_2\cdot |\eta(y)| \ dy \stackrel{\text{Lemma } \ref{auxilarylemma1}}{\leqslant} \left(\int_G\ \left(\sup_{\phi\in\Spher^+}\ \frac{\left|\phi(y) - 1\right|^2}{(1+\gamma(\phi)^2)^s}\right)^{\frac{1}{2}}\cdot |\eta(y)|\ dy\right) \cdot \|f_B\|_{H_{\gamma}^s}\\
&= \sup_{y\in\supp(\eta)}\ \left(\sup_{\phi\in\Spher^+}\ \frac{\left|\phi(y) - 1\right|}{(1+\gamma(\phi)^2)^\frac{s}{2}}\right)\cdot \|f_B\|_{H_{\gamma}^s} = \sup_{y\in\supp(\eta)}\ \left(\sup_{\phi\in\Spher^+}\ \frac{\left|\phi(y) - 1\right|}{(1+\gamma(\phi)^2)^\frac{s}{2}}\right)\cdot \|f\|_{H_{\gamma}^s},
\end{split}
\end{equation*}

\noindent
which ends the proof.
\end{proof}

Up to this point, $(G,K)$ was an arbitrary Gelfand pair with $G$ a locally compact Hausdorff group and $K$ its compact subgroup. Now we impose a further restriction, namely:
\begin{center}
\textit{we assume that $G$ is a compact group.}
\end{center}

\noindent
For brevity we will write that $(G,K)$ is a compact Gelfand pair, meaning that $G$ is compact. This assumption somewhat simplifies the ``dual object'' of $(G,K)$. In summary, if $G$ is compact then $\Spher(G,K) = \Spher^b(G,K) = \Spher^+(G,K)$ (comp. Theorem 9.10.1 in \cite{Wolf}, p. 204) and this is a compact space due to Theorem 9.1.13 in \cite{Wolf}, p. 183 and the fact that $L^1(\KGK)$ has a unit element. 

\begin{lemma}
Let $(G,K)$ be a compact Gelfand pair and let $p,q\in(1,\infty).$ If $(f_n)\subset L^p(\KGK)$ is a weakly convergent sequence (with limit function $f$), then for every $\eta\in C_c(\KGK)$ the sequence $(f_n\star\eta)$ converges (strongly) to $f\star\eta$ in $L^q(\KGK)$. 
\label{fnetaconvergesfeta}
\end{lemma}
\begin{proof}
Firstly, we fix a function $\eta\in C_c(\KGK).$ The sequence $(f_n)$ is a weakly convergent, so it is bounded in $L^p(\KGK)$, i.e. there exists $M>0$ such that $\|f_n\|_p\leqslant M$ for every $n\in\naturals$ (comp. Proposition 3.5 in \cite{Brezis}, p. 58). Furthermore, by definition of weak convergence, for every $x\in G$ we have
\begin{gather*}
f_n\star\eta(x) = \int_G\ f_n\left(xy^{-1}\right)\cdot \eta(y)\ dy \stackrel{y\mapsto yx}{=} \int_G\ f_n\left(y^{-1}\right)\cdot \eta(yx)\ dy \stackrel{y\mapsto y^{-1}}{=} \int_G\ f_n(y)\cdot \eta\left(y^{-1}x\right)\ dy\\
\stackrel{n\rightarrow\infty}{\longrightarrow} \int_G\ f(y)\cdot \eta\left(y^{-1}x\right)\ dy \stackrel{y\mapsto y^{-1}}{=}\int_G\ f\left(y^{-1}\right)\cdot \eta(yx)\ dy\stackrel{y\mapsto yx^{-1}}{=} \int_G\ f_n\left(xy^{-1}\right)\cdot \eta(y)\ dy = f\star\eta(x).
\end{gather*}

\noindent
In other words, the sequence $(f_n\star\eta)$ converges pointwise to $f\star\eta.$

Last but not least, for every $x\in G$ we have
\begin{equation*}
\begin{split}
|f_n\star\eta(x) - f\star\eta(x)| &= \left|\int_G\ \left(f_n\left(xy^{-1}\right) - f\left(xy^{-1}\right)\right)\cdot \eta(y)\ dy\right| \stackrel{y\mapsto yx}{=} \left|\int_G\ \left(f_n\left(y^{-1}\right) - f\left(y^{-1}\right)\right)\cdot \eta(yx)\ dy\right|\\ 
&\stackrel{y\mapsto y^{-1}}{=} \left|\int_G\ \left(f_n(y) - f(y)\right)\cdot \eta\left(y^{-1}x\right)\ dy\right| \stackrel{\text{H\"{o}lder ineq.}}{\leqslant} \|f_n - f\|_p\cdot \left(\int_G\ \left|\eta\left(y^{-1}x\right)\right|^{p'}\ dy\right)^{\frac{1}{p'}} \\
&\stackrel{y\mapsto xy}{\leqslant} 2M\cdot \left(\int_G\ \left|\eta\left(y^{-1}\right)\right|^{p'}\ dy\right)^{\frac{1}{p'}} \stackrel{y\mapsto y^{-1}}{=} 2M\cdot \|\eta\|_{p'}.
\end{split}
\end{equation*}

\noindent
Since $G$ is compact, then $2M\cdot \|\eta\|_{p'}$ is an integrable (with arbitrary power) dominating function for $|f_n\star\eta - f\star\eta|$. Finally, applying the Lebesgue dominated convergence theorem (comp. Theorem 4.2 in \cite{Brezis}, p. 90) we obtain 
$$\lim_{n\rightarrow\infty}\ \|f_n\star\eta - f\star\eta\|_q = 0,$$

\noindent
which concludes the proof.
\end{proof}

We are almost ready to prove Rellich-Kondrachov theorem for Gelfand pairs. One last, missing piece of the puzzle is Vitali convergence theorem:

\begin{thm}(comp. Theorem 7.13 in \cite{Bartle}, p. 76)\\
Let $(f_n)$ be a sequence in $L^p(X,\mu_X),$ where $1\leqslant p < \infty.$ The sequence $(f_n)$ converges to a function $f \in L^p(X,\mu_X)$ if and only if: 
\begin{enumerate}
	\item $(f_n)$ converges to $f$ in measure,
	\item for every $\eps>0$ there exists a measurable set $A_{\eps}$ with $\mu_X(A_{\eps}) < \infty$ and such that 
	$$\forall_{n\in\naturals}\ \int_{G\backslash A_{\eps}}\ |f_n(x)|^p\ dx < \eps^p,$$
	
	\item for every $\eps>0$ there exists $\delta>0$ such that for every measurable set $A$ with $\mu_X(A)<\delta$ we have
	$$\forall_{n\in\naturals}\ \int_A\ |f_n(x)|^p\ dx < \eps^p.$$
\end{enumerate}
\label{Vitaliconvergencetheorem}
\end{thm} 

At last, we gathered all the required tools to prove the culminating result of the current section, namely Rellich-Kondrachov theorem for compact Gelfand pairs. One last piece of terminology prior to the theorem itself: we say (comp. Definition 7.25 in \cite{RenardyRogers}, p. 211) that a Banach space $X$ is \textit{compactly embedded} in a Banach space $Y$ (which we denote $X\stackrel{c}{\hookrightarrow} Y$) if it is continuously embedded in $Y$ and if the embedding is a compact map (it maps bounded sets in $X$ to relatively compact sets in $Y$). 

\begin{thm}
Let $(G,K)$ be a compact Gelfand pair, $\alpha > s >0$ and let $p := \frac{2\alpha}{\alpha + s}.$ If 
$$(1+\gamma^2)^{-1}\in L^{\alpha}(\Spher^+,\widehat{\mu}),$$

\noindent
and
\begin{gather}
\lim_{y\rightarrow e}\ \left(\sup_{\phi\in\Spher^+}\ \frac{\left|\phi(y) - 1\right|}{(1+\gamma(\phi)^2)^\frac{s}{2}}\right) = 0,
\label{limitcondition}
\end{gather}

\noindent
then $H^s_{\gamma}(\KGK)\stackrel{c}{\hookrightarrow} L^q(\KGK)$ for every $q \in[1,p']$. In other words, $H^s_{\gamma}(\KGK)$ embeds compactly in $L^q(\KGK).$
\label{RellichKondrachov}
\end{thm}
\begin{proof}
Firstly we fix $q \in[1,p'].$ By Theorem \ref{Sobolevembeddingtheorem} we already know that $H^s_{\gamma}(\KGK)$ is continuously embedded in $L^{p'}(\KGK).$ Since $G$ is compact then $L^{p'}(\KGK)\hookrightarrow L^q(\KGK)$ and, as a consequence, $H^s_{\gamma}(\KGK)$ embedds continuously in $L^q(\KGK).$   

It remains to prove that the embedding is compact. By definition we need to check that every bounded set in $H^s_{\gamma}(\KGK)$ is mapped to a relatively compact set in $L^q(\KGK).$ Since we are dealing with metric (even Banach!) spaces, it suffices to prove that every bounded sequence $(f_n)\subset H^s_{\gamma}(\KGK)$ has a convergent subsequence in $L^q(\KGK)$. In other words, instead of checking topological compactness we check the equivalent (in this case) sequential compactness. 

We observe that (due to continuous embedding) a bounded sequence $(f_n)\subset H^s_{\gamma}(\KGK)$ is also bounded in $L^{p'}(\KGK).$ Consequently, we can choose a weakly convergent subsequence of $(f_n) -$ for notational convenience and clarity we still denote this convergent subsequence by $(f_n)$. Let $f\in L^{p'}(\KGK)$ be the weak limit of $(f_n).$ 

Fix $\eps > 0$ and let $\eta\in C_c(\KGK)$ be such that $\|f\star \eta - f\|_q < \eps$ (comp. Proposition 2.42 in \cite{Folland}, p. 53) and 
\begin{gather}
\sup_{y\in\supp(\eta)}\ \left(\sup_{\phi\in\Spher^+}\ \frac{\left|\phi(y) - 1\right|}{(1+\gamma(\phi)^2)^\frac{s}{2}}\right)\cdot M < \eps,
\label{wechooseeta}
\end{gather}

\noindent
where $M>0$ is a constant $H_{\gamma}^s(\KGK)-$bound of the sequence $(f_n)$. The latter condition can be satisfied due to (\ref{limitcondition}). For such a choice of $\eta$ we have
\begin{equation*}
\begin{split}
\|f_n - f\|_2 &\leqslant \|f_n - f_n\star\eta\|_2 + \|f_n\star\eta - f\star\eta\|_2 + \|f\star\eta - f\|_2\\
&\stackrel{\text{Lemma } \ref{auxilarylemma2}}{\leqslant} \sup_{y\in\supp(\eta)}\ \left(\sup_{\phi\in\Spher^+}\ \frac{\left|\phi(y) - 1\right|}{(1+\gamma(\phi)^2)^\frac{s}{2}}\right)\cdot \|f_n\|_{H^s_{\gamma}} + \|f_n\star\eta - f\star\eta\|_2 + \eps \\
&\stackrel{(\ref{wechooseeta})}{\leqslant} 2\eps + \|f_n\star\eta - f\star\eta\|_2,
\end{split}
\end{equation*}

\noindent
and since $\|f_n\star\eta - f\star\eta\|_2 \longrightarrow 0$ by Theorem \ref{fnetaconvergesfeta}, we conclude that $\lim_{n\rightarrow \infty}\ \|f_n - f\|_2\leqslant 2\eps.$ Since $\eps>0$ was chosen arbitrarily then $\lim_{n\rightarrow \infty}\ \|f_n - f\|_2 = 0.$

Since $f_n\rightarrow f$ in $L^2(\KGK)$ then the sequence $(f_n)$ also converges to $f$ in measure (comp. \cite{Royden}, p. 95). Furthermore, $G$ is compact so the second condition in Vitali convergence theorem is automatically satisfied. Last but not least, the sequence $(f_n)$ is uniformly $L^q(\KGK)-$integrable, so the third condition of Vitali convergence theorem is satisfied. We conclude that $(f_n)$ converges to $f$ in $L^q(\KGK).$ 
\end{proof}

\end{document}